\documentclass[a4paper, 11pt, reqno]{amsart}
\usepackage{amsmath,amsfonts,amssymb,amsthm,enumerate}
\usepackage{graphicx}
\usepackage{xy} \xyoption{all}

\usepackage{tikz,pgf,amscd,enumerate,multicol}
\usetikzlibrary{automata,arrows.meta,positioning}

\usepackage{color}

\usepackage[pagebackref,colorlinks]{hyperref}

\newtheorem{thm}{Theorem}[section]
\newtheorem{cor}[thm]{Corollary}
\newtheorem{prop}[thm]{Proposition}
\newtheorem{lem}[thm]{Lemma}

\newtheorem*{ques}{Question 4.7}

\theoremstyle{definition}
\newtheorem{defn}[thm]{Definition}

\newtheorem{rem}[thm]{Remark}
\newtheorem{quess}[thm]{Question}

\let\phi\varphi

\def\mG{\mathcal{G}}
\def\M{\mathbb{M}}

\newcommand{\K}{\mathsf{k}}

\newcommand{\gr}{\operatorname{gr}}

\begin{document}
	
\title{Embedding $\K$-algebras into Leavitt algebra $L_\K(1, 2)$}

\author{Boris Bilich}
\address{(Boris Bilich): Department of Mathematics, University of Haifa, Haifa, Israel.
}
\address{Mathematisches Institut, Universit\"at G\"ottingen, G\"ottingen,
Germany
}
\email{bilichboris1999@gmail.com}

\author{Roozbeh Hazrat}
\address{(Roozbeh Hazrat): Centre for Research in Mathematics and Data Science, Western Sydney University, 
			Australia.
}
\email{r.hazrat@westernsydney.edu.au}
	
\author{Tran Giang~Nam}
\address{(Tran Giang Nam): Institute of Mathematics, VAST, 18 Hoang Quoc Viet, Nghia Do, Hanoi, Vietnam.}
\email{tgnam@math.ac.vn}

\begin{abstract} Since the commutative monoid $T = (\{0, 1\}, \vee)$ is a weak terminal object in the category of  conical monoids with order units, there is a unital homomorphism from  every Bergman $\K$-algebra corresponding to a conical finitely generated commutative monoid   into the Leavitt algebra $L_\K(1,2)$, where $\K$ is a field. This fact will be used to give a short proof that Leavitt path algebras associated with finite graphs with condition $(L)$  embed into $L_\K(1,2)$, as well as provide criteria for an embedding of $\M_s(L_{\K}(1, m))$ in $\M_s(L_{\K}(1, n))$. As our second main result, we show that the Heisenberg equation $xy-yx=1$ cannot be realized in any Steinberg algebra, implying that the first Weyl algebra cannot be embedded into $L_\K(1,2)$, giving an affirmative answer to a question of  Brownlowe and S\o rensen on the embeddability of $\K$-algebras  with a countable basis inside $L_\K(1,2)$.  Whereas, $L_\K(E)$ cannot be graded-embedded into $L_\K(1,2)$ in general, in the final section we show that $L_\K(E)$ does admit a graded embedding into $L_\K(1,2)\otimes_\K L_\K(1,2)$.

\medskip

\end{abstract}

\maketitle

\section{Introduction}

 Let $\K$ be a field, and $m, n$ positive integers with $m < n$.  Let $A=(x_{ij})$ and $B=(y_{ji})$, where $1\leq i \leq m$ and $1\leq j \leq n$, be a collection of symbols considered as $m\times n$ and $n\times m$ matrices, respectively.  The Leavitt algebras are defined as follows:  
\begin{equation}\label{firstdef}
L_\K(m,n) \  := \  \K\big\langle A, B\big\rangle\big /  \big\langle A B = I_m,  \ BA=I_n \big\rangle,
\end{equation}
\noindent
where $\K\langle A, B\rangle$ is the free associative (noncommutative) unital ring generated by symbols $x_{ij}$ and $y_{ji}$ with coefficients in $\K$. Furthermore,  $AB=I_m$ and $BA=I_n$ stand for the collection of relations after multiplying the matrices and comparing the two sides of the equations. These algebras were considered by Leavitt \cite{leavitt1, leavitt2} in the context of classification of algebras without invariant basis number. Moreover, it was shown that $L_\K(1,n)$ is a (purely infinite) simple ring, while  $L_\K(m,n)$, $m>1$,  has no zero divisors.

Of particular interest is  the Leavitt algebra 
\[L_\K(1,2) = \K \langle x_1, x_2, x_1^*, x_2^* \rangle / \langle x_1 x_1^* + x_2 x_2^*=1, \,  x_1^* x_1= x_2^* x_2=1, \, x_1^*x_2 = x_2^*x_1 = 0\rangle,\]  which is  the discrete version of Cuntz algebra $\mathcal O_2$ and which has become a focal point of significant research interest. As an example, the infinite, finitely presented simple Thompson group $V$ is contained in the group of units of $L_\K(1,2)$~\cite{lpabook, sorensen}.  
Graph $C^*$-algebras have become their own industry in $C^*$-algebra theory and have recently been used in the ongoing classification of $C^*$-algebras program (see, e.g., \cite{rss2015}). Inspired by the success of graph $C^*$-algebras, Abrams and Aranda-Pino \cite{GenePino} and, independently, Ara, Moreno and Pardo \cite{ara2006} generalized Leavitt algebras $L_{\K}(1, n)$ by introducing Leavitt path algebras of directed graphs. One of the main topics in Leavitt path algebras has been a desire to get a purely algebraic version of the celebrated classification theorem of Kirchberg and Phillips.

In light of the Kirchberg's celebrated embedding Theorem that all separable, exact $C^*$-algebras embed into $\mathcal O_2$, in~\cite[Theorem 4.1]{brownsor}  Brownlowe and S\o rensen proved that any Leavitt path algebra of a countable graph can be embedded into $L_\K(1,2)$, and they raised the following question~\cite[p.~349]{brownsor}.  

\begin{ques}\label{sorensen}
Let $\K$ be a field. Does there exist a $\K$-algebra $A$ with a countable basis
such that $A$ does not embed into $L_\K(1,2)$?
\end{ques}

They noted ``The authors have been unable to come up with any such $A$, but we would find an
affirmative answer to the above question quite surprising''.

In this article,  we give an affirmative answer to Question 4.7, by showing that the first Weyl algebra  $\mathbb A_\K=\K\langle x, y\rangle/\langle xy-yx- 1\rangle$ cannot be embedded into the Leavitt algebras $L_\K(m,n)$ (Theorem \ref{globalmain}). This naturally raises the question of which classes of algebras can and cannot be embedded into $L_\K(1,2)$. Note that the Weyl algebra is a left/right Noetherian simple ring having no zero divisors and IBN property. 

Also, using Bergman's machinery~\cite{Berg74}, we give criteria for an embedding of $\M_s(L_{\K}(1, m))$ into $\M_s(L_{\K}(1, n))$ (Theorem \ref{LAlg-Embed} and Corollary \ref{LAlg-Embed1}), and provide a short proof that Leavitt path algebras associated with finite graphs with condition $(L)$  embed into $L_\K(1,2)$ (Theorem \ref{sLPA-embed}).

Given that Leavitt path algebras have a canonical $\mathbb Z$-grading which plays a pivotal role in their structures, we consider a graded version of Kirchberg embedding theorem. Whereas, $L_\K(E)$ cannot be graded-embedded into $L_\K(1,2)$ in general (e.g., there is no graded unital homomorphism $L_\K(1,3) \rightarrow L_\K(1,2)$), using Brownlowe and S\o rensen's $*$-preserving embedding of Leavitt path algebras of countable row-finite graphs into $L_\K(1, 2)$ (\cite[Theorem 4.1]{brownsor}), we obtain that these algebras can be graded-embedded into $L_\K(1,2)\otimes_\K L_\K(1,2)$ (Theorem \ref{grad-embed}), where 
the grading given by usual canonical $\mathbb Z$-grading on the first component and 
the trivial grading on the second component.

\section{Unital embeddings of Bergman algebras}
 
Brownlowe and S\o rensen's Theorem~\cite{brownsor} proved that  Leavitt path algebras embed into $L_\K(1,2)$. Their approach is to produce infinite number of orthogonal idempotents inside $L_\K(1,2)$ which match the structure of Leavitt path algebras $L_\K(E)$, for any graph $E$. Using Bergman's monoid realization work~\cite{Berg74}, we first 
provide criteria for an embedding of $\M_s(L_{\K}(1, m))$ in $\M_s(L_{\K}(1, n))$ (Theorem \ref{LAlg-Embed} and Corollary \ref{LAlg-Embed1}), and give a very short proof of embedding of $L_\K(E)$ inside $L_\K(1,2)$ when the graphs have condition (L), i.e., cycles have exits (Theorem \ref{sLPA-embed}). Consequently, we obtain that the Cohn path algebra of a finite graph can be embedded into the Leavitt algebra $L_{\K}(1, 2)$ (Corollary \ref{CPA-embed}).
  

We begin this section by recalling some basic notions of the commutative monoid theory.

\begin{defn}[{cf. \cite[Definition 2.5]{ANP2017}}]\label{monoiddefs}
Let $M$ be a commutative monoid (i.e., $M$ is a set, and $+$ is an associative commutative binary operation on $M,$ with a neutral element).   
	
(1)  We define a relation $\leq$ on $M$ by setting
$$x\leq y \ \ \mbox{in case there exists }  z\in M  \ \mbox{for which} \ x+ z = y.$$ Then $\leq$ is a preorder (reflexive and transitive).  
	
(2) $M$ is called {\it conical} if $x + y = 0$ implies  $x=y=0$ for all $x, y\in M$.
	
(3) An element $d\in M$ is called an \emph{order-unit} if, for any $x\in M,$  there exists a positive integer $n$ such that $x\leq nd$.	

(4) Let $(M, d)$ and $(M', d')$ be commutative monoids with order-units $d$ and $d'$, respectively. A {\it homomorphism} from $(M, d)$ to $(M', d')$ is a monoid homomorphism $f: M\longrightarrow M'$ satisfying $f(d) = d'$.
\end{defn}

For any unital  ring $R$, we denote by $\mathcal{V}(R)$ the set of isomorphism classes (denoted by $[P]$) of finitely generated projective right $R$-modules, and we endow $\mathcal{V}(R)$ with a structure of a commutative monoid by imposing the operation $$[P] + [Q] = [P\oplus Q]$$ for isomorphism classes $[P]$ and $[Q]$. 
It is not hard to check that $\mathcal{V}(R)$ is a conical monoid and $[R]$ is an order-unit in $\mathcal{V}(R)$. If $\phi: R\longrightarrow S$ is a homomorphism of unital rings, then $\phi$ induces a well-defined \emph{pointed} monoid homomorphism $$\mathcal{V}(\phi): (\mathcal{V}(R), [R]) \longrightarrow (\mathcal{V}(S), [S])$$ such that $\mathcal{V}(\phi)([P]) = [S\otimes_RP]$ for all $[P]\in \mathcal{V}(R)$. 

We recall that a unital ring $R$ is called {\it right hereditary} if every right ideal of $R$ is projective.

A fundamental and striking theorem of George Bergman \cite{Berg74} shows that a finitely generated conical commutative monoid can be realized as the $\mathcal{V}$-monoid of a (left and right) hereditary ring which has a certain universal property: 

\begin{thm}[{\cite[Theorem 6.2]{Berg74}}]\label{Bergman}
Let $\K$ be a field and $M$	a finitely generated conical commutative monoid with an order-unit $d$. Then, there exists a (left and right) hereditary unital $\K$-algebra $B$ such that $(\mathcal{V}(B), [B])\cong (M, d)$. Moreover, $B$ can be taken to have the weak universal property that for any $\K$-algebra $A$ and any homomorphism $\psi: (M, d)\longrightarrow (\mathcal{V}(A), [A])$, there exists a $\K$-algebra homomorphism $\phi: B \longrightarrow A$ such that $\mathcal{V}(\phi) = \psi$.
\end{thm}

We say that the hereditary $\K$-algebra  $B$, introduced in Theorem \ref{Bergman}, is the {\it Bergman $\K$-algebra corresponding to $(M, d)$} (\cite{Genenotices,lpabook}).

As a first application of Theorem~\ref{Bergman}, we obtain an algebraic generalization of Kawamura's result~\cite[Lemma 2.1]{Kawa} on the embedibilty $\mathcal O_m \rightarrow \mathcal O_n$ and its generalization $\mathcal O_m \rightarrow \M_t(\mathcal O_n)$ by Hajac and Liu in \cite[Theorem 4.5]{HajLiu}.

Let $\K$ be a field, and $m$ and $m'$ positive integers with $m < m'$. Let $M_{m, m'}$ be the commutative (monogenic) monoid
\begin{center}
$M_{m, m'} := \{0, x, \hdots, mx, \hdots, (m'-1)x\}$ with relation $mx = m'x$.
\end{center}
That means, $$M_{m, m'}\cong \mathbb{N}/\langle m =m'\rangle,$$ 
where $\langle m =m'\rangle$ is the congruence on $\mathbb{N}$ generated by $(m, m')$ (see \cite[Chapter~1]{howie}). 
We call the elements $\{0,x,\hdots, (m-1)x\}$ the \emph{tail} of the monoid and $\{mx, \hdots, (m'-1)x\}$ the \emph{periodic} part of $M_{m, m'}$ which coincides with the group completion of
the monoid $M_{m,m'}$. Note that for $s,t \in \mathbb N$, if $s=t$ in $M_{m,m'}$, then $s=t<m$ (i.e., they appear in the tail) or $m'-m\mid s-t$.

For any positive integer $s$, by \cite[Theorems 6.1 and 7.3]{Berg74} (see also \cite[Theorem~5.7]{GeneRooz}), we have 
\begin{equation}\label{monformula}
    \left (\mathcal{V}(\M_s(L_{\K}(m, m'))), [\M_s(L_{\K}(m, m'))]\right )\cong (M_{m, m'}, s)
\end{equation}
 and $\M_s(L_{\K}(m, m'))$ has the weak universal property. 
 
Using these notes and Theorem \ref{Bergman}, we obtain the following interesting result.

\begin{thm}\label{LAlg-Embed}
Let $\K$ be a field and let $n <  n'$, $m< m'$, $s$ and $t$ be positive integers. Then, there is a unital homomorphism \[\phi: \M_s(L_{\K}(m,m')) \longrightarrow \M_t(L_{\K}(n,n'))\] if and only if there exist positive integers $k$ and $k'$ such that 
$ks=t<n$, or
$ks\ge n,\ t\ge n$, as well as  $ks\equiv t \mod(n'-n)$, 
 \(km\ge n\), and \( (m' - m)k = (n' - n)k'\).

\begin{proof}
($\Longrightarrow$) 
Using~(\ref{monformula}), the ring homomorphism $\phi$ induces a monoid homomorphism
\(\overline \varphi:M_{m,m'}\to M_{n,n'}\) such that $\overline \phi(s)=t$. Let $\overline \phi(1)=k$, for some \(k> 0\).
Then
$ks=\overline \varphi(s)=t$ in the monoid $M_{n,n'}$. If $t<n$, then $ks=t$. Otherwise $t\geq n$ and thus $ks\ge n$ and 
$ks\equiv t \mod(n'-n)$. Since $m=m'$ in $M_{m,m'}$, then $km=\overline\phi(m)=\overline\phi(m')=km'$. Thus $km\geq n$ and $(m' - m)k = (n' - n)k'$.

($\Longleftarrow$) If the conditions in the Theorem holds, the assignment
$\overline \varphi:M_{m,m'}\to M_{n,n'}; 1\mapsto k$ defines a well-defined homomorphism with
\(\bar \varphi(s)=t\). Since the algebra $\M_s(L_\K(m,m'))$ is the Bergman algebra which realizes $M_{m,m'}$ with the order unit $s$,  
the homomorphism $\phi$ follows from
Theorem~\ref{Bergman}, thus finishing the proof.
\end{proof}    
\end{thm}

Consequently, we have the following corollary.

\begin{cor}\label{LAlg-Embed1}
Let $\K$ be a field and let $n,\, m,\, s$ and $t$ be positive integers such that $1 < m$ and $1 < n$. Then, there is a unital injective homomorphism \[\M_s(L_{\K}(1, m)) \longrightarrow \M_t(L_{\K}(1, n))\] if and only if there exist positive integers $k$ and $k'$ such that $(m - 1)k = (n - 1)k'$ and $ks \equiv t \mod{(n - 1)}$.
\end{cor}
\begin{proof}
It immediately follows from Theorem \ref{LAlg-Embed} and the fact that $L_{\K}(1, m)$ is simple.
\end{proof}

A (directed) graph $E = (E^0, E^1, s, r)$ 
consists of two disjoint sets $E^0$ and $E^1$, called \emph{vertices} and \emph{edges}
respectively, together with two maps $s, r: E^1 \longrightarrow E^0$.  The
vertices $s(e)$ and $r(e)$ are referred to as the \emph{source} and the \emph{range}
of the edge~$e$, respectively. A graph $E$ is \emph{finite} if both sets $E^0$ and $E^1$ are finite.  A graph $E$ is \emph{row-finite} if $|s^{-1}(v)| < \infty$ for all $v\in E^0$. A vertex~$v$ is \emph{regular}
if $0 < |s^{-1}(v)| < \infty$.
A \emph{path} $p = e_{1} \cdots e_{n}$ in a graph $E$ is a sequence of
edges $e_{1}, \dots, e_{n}$ such that $r(e_{i}) = s(e_{i+1})$ for $i
= 1, \dots, n-1$.  
A path $p$ is called a \emph{cycle} if $s(p) = r(p)$, and $p$ does not revisit
any other vertex. 
An edge~$f$ is an \emph{exit} for a path $p = e_{1} \cdots e_{n}$ if $s(f) = s(e_{i})$
but $f \ne e_{i}$ for some $1 \le i \le n$. We say that $E$ satisfies {\it condition} $(L)$ if every cycle in $E$ has an exit.

Let $E=(E^0,E^1,r,s)$ be a finite  graph. The {\it graph monoid} $M_E$ is defined as the free commutative monoid on a set of generators $E^0$, modulo relations by
\[ v = \sum_{e\in s^{-1}(v)} r(e)\]
for each regular vertex $v$. 

The notion of a Cohn path algebra has been defined and investigated by Ara and Goodearl \cite{AraGoodearl}
(see, also, \cite{lpabook}). Specifically, for an arbitrary graph $E = (E^0,E^1,s,r)$
and an arbitrary unital ring $R$, the \emph{Cohn path algebra} $C_{R}(E)$ of the graph~$E$
\emph{with coefficients in}~$R$ is the $R$-algebra generated
by the sets $E^0$ and $E^1$, together with a set of variables $\{e^{*}\ |\ e\in E^1\}$,
satisfying the following relations for all $v, w\in E^0$ and $e, f\in E^1$:
\begin{itemize}
\item[(1)] $v w = \delta_{v, w} w$;
\item[(2)] $s(e) e = e = e r(e)$ and
$r(e) e^* = e^* = e^*s(e)$;\item[(3)] $e^* f = \delta_{e, f} r(e)$,
\end{itemize}
where $\delta$ is the Kronecker delta. 

Let $I$ be the ideal of $C_{R}(E)$ generated by all elements of the form
$v-\sum_{e\in s^{-1}(v)}ee^*$, where $v$ is a regular vertex. Then the $R$-algebra $C_{R}(E)/I$ is called the \emph{Leavitt path algebra}
of $E$ with coefficients in $R$, denoted by $L_{R}(E)$.

Typically  the Leavitt path algebra $L_{R}(E)$ is defined without reference to Cohn
path algebras, rather, it is defined as the $\K$-algebras are generated by
the set $\{v, e, e^*\ |\ v\in E^0, e\in E^1\}$,
which satisfy the  above conditions (1), (2), (3), and the additional condition:
\begin{itemize}
\item[(4)] $v= \sum_{e\in s^{-1}(v)}ee^*$ for any  regular vertex $v$.
\end{itemize}
\medskip

We are in a position to show that Leavitt path algebras associated to graphs with condition $(L)$
embed into $L_\K(1,2)$. This re-establishes the Brownlowe and S\o rensen result more directly (but does not explicitly give the injective homomorphism).

\begin{thm}\label{sLPA-embed}
Let $\K$ be a field, $R$ a unital $\K$-algebra and $E$ be a finite graph with condition $(L)$. Then, $L_R(E)$ can be embedded into the Leavitt algebra $L_R(1, 2)$. 
\end{thm}
\begin{proof}
Let $M$	be a finitely generated conical commutative monoid with an order-unit $d$, and $B$ the Bergman $\K$-algebra corresponding to $(M, d)$.
By Theorem \ref{Bergman}, we have $(\mathcal{V}(B), [B])\cong (M, d)$, and  $B$ has  the weak universal property. 
Consider the commutative monoid $T = (\{0, 1\}, \vee)$. It is obvious that $T$ is finitely generated and conical, and $1$ is the order-unit of $T$. Consider the map $\psi: M\longrightarrow T$ defined by $\psi(0) =0$ and $\psi(x) =1$ for all $x\in M\setminus \{0\}$. Let $x$ and $y$ be two elements of $M$. If $x + y = 0$, then since $M$ is conical, $x =0 =y$, and so $\psi(x + y) = 0 = \psi(x) \vee \psi(y)$. If $x + y \neq 0$, then we have either $x\neq 0$ or $y\neq 0$, and so $\psi(x)= 1$ or $\psi(y)=1$. This shows that $\psi(x) \vee \psi(y) = 1 = \psi(x + y)$. Therefore, $\psi$ is a monoid homomorphism with $\psi(d) =1$, that means, $\psi$ is a pointed homomorphism from $(M, d)$ to $(T, 1)$. Moreover, by \cite[Theorem 3.2.5]{lpabook}, we have $(L_\K(1, 2), [L_\K(1, 2)]) \cong (T, 1)$. By the weak universal property of $B$, there exists a $\K$-algebra homomorphism $\phi: B\longrightarrow L_\K(1, 2)$ such that $\mathcal{V}(\phi) = \psi$.

Specialising  the conical monoid $M$ to the graph monoid $M_E$, by \cite[Theorem~3.5]{ara2006}, we have  $$\left (\mathcal{V}(L_\K(E), [L_\K(E)]\right )\cong (M_{E}, d_{E}),$$ where $ d_{E} = \sum_{v\in E^0}v$, and $L_\K(E)$ is a Bergman algebra for $M_E$ and thus has 
the weak universal property. Thus, by the above observation, we have a unital homomorphism $\phi: L_\K(E) \longrightarrow L_\K(1, 2)$.

We claim that $\phi$ is injective. Indeed, let $a$ be a nonzero idempotent element of $L_\K(E)$. Since the right ideal $aL_\K(E)$ is a nonzero finitely generated projective right $L_\K(E)$-module, we have $$\mathcal{V}(\phi)([aL_\K(E)]) = \psi([aL_\K(E)])= 1,$$ and so 
\[\mathcal{V}(\phi)([aL_\K(E)]) = [aL_\K(E)\otimes_{L_\K(E)} L_\K(1, 2)]\neq 0.\] 

On the other hand, since the exact sequence $$0 \to  a L_\K(E)
\xrightarrow{\iota} L_\K(E) \xrightarrow{\pi} L_\K(E)/ a L_\K(E)\to 0$$ of right $L_\K(E)$-modules, where $\iota$ is the canonical injection and $\pi$ is the canonical surjection, is split, we have $[\phi(a)L_\K(1, 2)]=[aL_\K(E)\otimes_{L_\K(E)} L_\K(1, 2)]\neq 0.$ This implies that $\phi(a)\neq 0$. In particular, $\phi(v)\neq 0$ for all $v\in E^0$. 
By our hypothesis that $E$ satisfies Condition
(L), we may apply the Cuntz-Krieger Uniqueness Theorem~\cite[Theorem 2.2.16]{lpabook} to conclude that  $L_\K(E)$ embeds into $L_\K(1, 2)$.

Since $R$ is flat over $\K$, we have an injection $L_\K(E) \otimes_\K R \hookrightarrow L_\K(1,2) \otimes_\K R$, hence,   $L_R(E)$ embeds into $L_R(1, 2)$ as well, thus finishing the proof.
\end{proof}

We should mention that Corti\~ nas and Montero \cite[Corollary 3.2]{Cortinas2020} showed that any simple Leavitt path $\K$-algebra of a finite graph embeds into the Leavitt algebra $L_{\K}(1, 2)$ using $K$-theory.

Consequently, we will obtain that the Cohn path algebra of a finite graph can be embedded into the Leavitt algebra $L_{\K}(1, 2)$.  Before doing so, we represent a specific case of a more general result described in \cite[Section 1.5]{lpabook}. Namely,
let $E = (E^0, E^1, s, r)$ be an arbitrary graph and $Y$ the set of regular vertices of $E$.
Let $Y'=\{v'\ |\ v\in Y\}$ be a disjoint copy of $Y$. For $v\in Y$ and for each edge $e$ in $E^1$ such
that $r_E(e)=v$, we consider a new symbol $e'$. We define the graph $F(E)$ as follows:
$$F(E)^0 := E^0\sqcup Y' \text{ and } F(E)^1 :=E^1\sqcup \{e'\ |\ r_E(e)\in
Y\},$$ and for each $e\in E^1$, $s_{F(E)}(e)=s_E(e),\ s_{F(E)}(e')=s_E(e),\
r_{F(E)}(e)=r_E(e)$, and $r_{F(E)}(e')=r_E(e)'$. For instance, if

\begin{tikzpicture}[>=stealth, scale=1.15, every node/.style={scale=1.15}]

\node[circle, fill, inner sep=1pt] (v) at (0,0) {};
\node at (-0.4,0) {$v$};

\draw[->, thick, out=140, in=60, looseness=50]
  (v) to (v) node[above, yshift=23pt] {$e$};

\draw[->, thick, out=-60, in=-140, looseness=50]
  (v) to (v) node[ below, yshift=-22pt] {$f$};

\node at (-2,0) {$E$};

\node[circle, fill, inner sep=1pt] (v1) at (4.6,0) {};
\node[circle, fill, inner sep=1pt] (v2) at (7.2,0) {};

\node at (4.3,0) {$v$};
\node at (7.7,0) {$v'$};

\draw[->, thick, out=140, in=60, looseness=50]
  (v1) to (v1) node[above, yshift=23pt] {$e$};

\draw[->, thick, out=-60, in=-140, looseness=50]
  (v1) to (v1) node[below, yshift=-22pt] {$f$};

\draw[->, thick, bend left=22] (v1) to node[above] {$e'$} (v2);
\draw[->, thick, bend right=22] (v1) to node[below] {$f'$} (v2);

\node at (3.2,0) {$F(E)$};

\end{tikzpicture}

In \cite[Theorem 1.5.17]{lpabook}, Abrams, Ara and Siles Molina showed that for any field $\K$
and any graph $E$, there is an isomorphism of $\K$-algebras $C_\K(E)\cong L_\K(F(E))$. Using this surprising result and Theorem \ref{sLPA-embed} we get the following.

\begin{cor}\label{CPA-embed}
Let $\K$ be a field, $R$ a unital $\K$-algebra and $E$ a finite graph. Then, the Cohn path algebra $C_{R}(E)$ can be embedded into the Leavitt algebra $L_{R}(1, 2)$.
\end{cor}
\begin{proof}
We have that $C_R(E)\cong L_R(F(E))$, where $F(E)$ is the finite graph constructed from $E$ cited above. Using that discription, it is easy to see that $F(E)$ is a finite graph with condition $(L)$. By Theorem \ref{sLPA-embed}, $C_R(E)\cong L_R(F(E))$ embeds into the Leavitt algebra $L_{R}(1, 2)$, thus finishing the proof.
\end{proof}

\begin{rem}
Theorem~\ref{sLPA-embed} extends to Leavitt path algebras associated to row-finite graphs, by noting that such an algebra can be written as a 
direct limit of Leavitt path algebras of finite graphs~\cite{ara2006} and that Bergman's construction carries over to direct limits (see~\cite[Theorem~6.4]{Berg74}). 
\end{rem}

\begin{quess}
Among the generalizations of Leavitt path algebras are the algebras arising from higher-rank graphs, referred to as Kumjian-Pask algebras~\cite{Pino}. It is not yet known whether Kumjian-Pask algebras can be embedded into $L_\K(1,2)$. It is also not clear which class of Kumjian-Pask algebras can be realized as Bergman algebras.  If affirmative, this in particular answers the question whether $L_\K(1,2) \otimes_{\K} L_\K(1,2)$ embeds into $L_\K(1,2)$ as $L_\K(1,2) \otimes_{\K} L_\K(1,2)$ can be realized as the Kumjian-Pask algebra of the $2$-graph
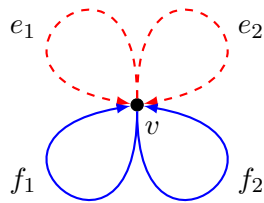
\[
\begin{tikzpicture}[scale=1]
\node[circle,draw,fill=black,inner sep=0.5pt] (p11) at (0, 0) {$.$} 

edge[-latex, red,thick,loop, dashed, in=10, out=90, min distance=70] (p11)
edge[-latex, red,thick, loop, dashed, out=90, in=170, min distance=70] (p11)

edge[-latex, blue,thick,loop,  in=-10, out=-90, min distance=70] (p11)

edge[-latex, blue,thick, loop, out=-90, in=-170, min distance=70] (p11);

\node at (0.2, -0.3) {$v$};
\node at (-1.5,1) {$e_1$};
\node at (1.5,1) {$e_2$};
\node at (1.5,-1) {$f_2$};
\node at (-1.5,-1) {$f_1$};
\end{tikzpicture}
\]
in which the blue and red edges commute with one another (see ~\cite{arawillie, sorensen}). 
\end{quess}

\section{An affirmative answer to  Brownlowe and S\o rensen's question}

In this section, we show that the first Weyl algebra $\mathbb A_\K$ does not embed into any unital Steinberg algebra (Theorem \ref{main}). Consequently, we obtain that $\mathbb A_\K$ does not embed into $L_\K(1, 2)$, which gives an affirmative answer to a question of  Brownlowe and S\o rensen Question~~\cite[p.~349]{brownsor}.  We further show that $\mathbb A_\K$ does not embed into $L_\K(m,n)$ for any pair $m < n$ (Theorem \ref{globalmain}). 
\medskip
 
We will use the following result of Wintner-Wielandt \cite{Wielandt} (see also Halmos~\cite[Chapter~24]{halmos}) showing that in a Banach algebra the additive commutator $PQ-QP$ cannot be a scalar, which plays an important role in our analysis. 

\begin{prop}\label{wielandt}
Let $A$ be a Banach algebra with identity over real or complex numbers. If $PQ - QP= \alpha 1_A$, where $P,Q\in A$ and $\alpha $ a scalar, then $\alpha  = 0$. 
\end{prop}
\begin{proof}
For the reader's convenience, we briefly sketch its proof here. 

Assume that $PQ - QP= \alpha 1_A$. Then, by induction on $n$, we immediately obtain that $$P^n Q - QP^n = n P^{n-1}\alpha$$ for all positive integer $n$. If $P^{n-1}\neq 0$ but $P^n=0$ for some $n \ge 1$, then $n P^{n-1}\alpha =0$, and so $\alpha =0$. If $P^n \neq 0$ for all $n\ge 1$, then by the triangle inequality we have $$n||P^{n-1}||\cdot |\alpha| = ||n P^{n-1}\alpha|| = ||P^n Q - QP^n||\le 2 ||P||\cdot ||Q||\cdot ||P^{n-1}||$$ for all $n$. This implies that $n|\alpha| \le 2||P||\cdot ||Q||$ for all $n$, and so $\alpha =0$, thus finishing the proof.
\end{proof}

We next recall the concepts of ample groupoids and Steinberg algebras. A \textit{groupoid} is a small category in which every morphism is invertible. It can also be viewed as a generalization of a group which has a partial binary operation. 
Let $\mathcal{G}$ be a groupoid. If $x \in \mathcal{G}$, $s(x) = x^{-1}x$ is the source of $x$ and $r(x) = xx^{-1}$ is its range. The pair $(x,y)$ is composable if and only if $r(y) = s(x)$. The set $\mathcal{G}^{(0)}:= s(\mathcal{G}) = r(\mathcal{G})$ is called the \textit{unit space} of $\mathcal{G}$. Elements of $\mathcal{G}^{(0)}$ are units in the sense that $xs(x) = x$ and
$r(x)x = x$ for all $x\in \mathcal{G}$. For $U, V \subseteq \mathcal{G}$, we define
\begin{center}
	$UV = \{\alpha\beta \mid \alpha\in U, \beta\in V \text{\, and \,} r(\beta) = s(\alpha)\}$ and $U^{-1} = \{\alpha^{-1}\mid \alpha\in U\}$.
\end{center}

A \textit{topological groupoid} is a groupoid endowed with a topology under which the inverse map is continuous, and such that the composition is continuous with respect to the relative topology on $\mathcal{G}^{(2)} := \{(\beta, \gamma)\in \mathcal{G}^2\mid s(\beta) = r(\gamma)\}$ inherited from $\mathcal{G}^2$. An \textit{\'{e}tale groupoid} is a topological groupoid $\mathcal{G},$ whose unit space  $\mathcal{G}^{(0)}$ is locally compact Hausdorff, and such that 
the domain map $s$ is a local homeomorphism. In this case, the range map $r$ and the multiplication map are local homeomorphisms and  $\mathcal{G}^{(0)}$ is open in $\mathcal{G}$ \cite{r:egatq}. 

An \textit{open bisection} of $\mathcal{G}$ is an open subset $U\subseteq \mathcal{G}$ such that $s|_U$ and $r|_U$ are homeomorphisms onto an open subset of $\mathcal{G}^{(0)}$.  Similar to \cite[Proposition 2.2.4]{p:gisatoa} we have that $UV$ and $U^{-1}$ are compact open bisections for all compact open bisections $U$ and $V$ of an  \'{e}tale groupoid  $\mathcal{G}$.
An \'{e}tale groupoid $\mathcal{G}$ is called \textit{ample} if  $\mathcal{G}$ has a base of compact open bisections for its topology.

Steinberg algebras were introduced in \cite{stein:agatdisa} in the context of discrete inverse semigroup algebras. Let $\mathcal{G}$ be an ample groupoid, and $R$ a unital commutative ring with the discrete topology. We denote by $R^{\mathcal{G}}$ the set of all continuous functions from $\mathcal{G}$ to $R$. Canonically,  $R^{\mathcal{G}}$ has the structure of a $R$-module with
operations defined pointwise.

\begin{defn}[{The Steinberg algebra}] 
	Let $\mathcal{G}$ be an ample groupoid, and  $R$ a unital commutative ring.  Let $A_R(\mathcal{G})$ be the $R$-submodule of $R^{\mathcal{G}}$ generated by the set
	\begin{center}
		$\{1_U\mid U$ is a compact open bisection of $\mathcal{G}\}$,
	\end{center}
	where $1_U: \mathcal{G}\longrightarrow R$ denotes the characteristic function on $U$. The \textit{multiplication} of $f, g\in A_R(\mathcal{G})$ is given by the convolution
	\[(f\ast g)(\gamma)= \sum_{\gamma = \alpha\beta}f(\alpha)g(\beta)\] for all $\gamma\in \mathcal{G}$. The $R$-submodule $A_R(\mathcal{G})$, with convolution, is called the \textit{Steinberg algebra} of $\mathcal{G}$ over $R$.
\end{defn}

It is useful to note that \[1_U\ast 1_V = 1_{UV}\] for compact open bisections $U$ and $V$ (see \cite[Proposition 4.5]{stein:agatdisa}), and the Steinberg algebra $A_R(\mathcal{G})$ is unital if and only if $\mathcal{G}^{(0)}$ is compact (see \cite[Proposition 4.11]{stein:agatdisa}).

The proposition provides necessary conditions for Steinberg algebras to be domains, which is very useful in our below arguments.

\begin{prop}\label{domain}
Let $R$ be a unital commutative ring and $\mG$ an ample groupoid. If $A_R(\mG)$ has no zero divisors, then $R$ is a domain and $\mG$ is a torsion-free group with the discrete topology. In particular, $A_R(\mG)$ is isomorphic to the group ring $R\mG$.
\end{prop}
\begin{proof}
Suppose  $A_R(\mG)$	has no zero divisors, and assume that $\mathcal{G}^{(0)}$ contains two distinct elements $u$ and $v$. Since  $\mathcal{G}^{(0)}$ is Hausdorff, there exist two open subsets $U$ and $V$
of  $\mathcal{G}^{(0)}$ such that $u\in U$, $v\in V$ and $U \cap V = \emptyset$. Since $\mathcal{G}$ is an ample groupoid, $\mathcal{G}$ has a base of compact open bisections for its topology, and so there exist two compact open bisections $A$ and $B$ such that $u\in A\subseteq U$ and $v\in B\subseteq V$. We then have $1_A  1_B = 1_{A \cap B} = 1_{\emptyset} =0$, and hence, $A_R(\mathcal{G})$ has zero divisors,  a contradiction. Therefore,  $\mathcal{G}^{(0)}$ is a singleton set, and so  $\mathcal{G}$ is a group with the discrete topology. This implies that $A_R(\mG)$ is isomorphic to the group ring $R\mG$, and so  $R\mG$ has no zero divisors. Since $R$ is a subring of $R\mG$, $R$ is also a domain. If $\mG$ has an element $g$ of finite order $n\ge 2$, then $$(g-1)(1 + g + \cdots + g^{n-1})=0 \text{ in } R\mG$$ shows that $R\mG$ has zero divisors, a contradiction. Therefore, $\mG$ is a torsion-free group. 
\end{proof}

If $\mG$ is the trivial group, then $A_R(\mG) \cong R$ as $R$-algebras. In light of this note, we have the following definition.

\begin{defn}
Let $R$ be a unital commutative ring and $\mG$ an ample groupoid. We say that the Steinberg algebra $A_R(\mG)$ is {\it nontrivial}	if $A_R(\mG)\ncong R$.
\end{defn}

Steinberg algebras are discrete analogues of groupoid $C^*$-algebras (\cite{p:gisatoa} and \cite{renault}). We recall the construction of the full $C^*$-algebra of an ample groupoid as described in \cite{ClarkZim24}. Let $\mathcal{G}$ be an ample groupoid. Then $A_{\mathbb{C}}(\mathcal{G})$ is a $*$-algebra with involution $f^*(\gamma) = \overline{f(\gamma^{-1})}$. For each $x\in \mathcal{G}^{0}$, we consider a $*$-representation $\pi_x: A_{\mathbb{C}}(\mathcal{G}) \longrightarrow \mathcal{B}(\ell^2(\mathcal{G}_x))$ given on the
orthonormal basis $\delta_{\gamma}$, $\gamma \in \mathcal{G}_x:= \{\alpha \in \mathcal{G}\mid s(\alpha) = x\}$, by $$\pi_x(f)\delta_{\gamma} = \sum_{\alpha \in s^{-1}(r(\gamma))}f(\alpha)\delta_{\alpha\gamma}.$$
We note that $||\pi_x(1_U)||\le 1$ for all compact open bisection $U \subseteq \mathcal{G}$ (see, e.g. \cite[page 10]{steinwyk}). It follows that if $f = \sum^n_{i=1}k_i1_{U_i}\in A_{\mathbb{C}}(\mathcal{G})$, where each $U_i$ is a compact open bisection of $\mathcal{G}$, then $$||\pi_x(f)|| \le \sum^n_{i=1}|k_i|,$$ and so $||\pi_x(f)|| \le |f|$, where $| \cdot |$ is the norm defined in~\cite{steinwyk} by
$$|f| = \inf\Big\{\sum^n_{i=1}|c_i| \mid f =\sum^n_{i=1}c_i1_{U_i}, U_i \in \Gamma_c(\mathcal{G})\Big\},$$ 
where $\Gamma_c(\mathcal{G})$ is the set of all  compact
open bisections of $\mathcal{G}$. Let $\mathcal{H} = \oplus_{x\in \mathcal{G}^{(0)}} \ell^2(\mathcal{G}_x)$. Then there  is an injective $*$-homomorphism $\pi: A_{\mathbb{C}}(\mathcal{G}) \longrightarrow \mathcal{B}(\mathcal{H})$ such that 
\begin{center}
$\pi(f)((g_x)_{x\in \mathcal{G}^{(0)}}) = (\pi_x(f)g_x)_{x\in \mathcal{G}^{(0)}}$	and $||\pi(f)|| = \sup_{x\in \mathcal{G}^{(0)}}||\pi_x(f)||$ 
\end{center}
for all $(g_x)_{x\in \mathcal{G}^{(0)}}\in \mathcal{H}$ and $f\in A_{\mathbb{C}}(\mathcal{G})$. We call $\pi$ the {\it left regular representation} of $A_{\mathbb{C}}(\mathcal{G})$. By \cite[Theorem 4.7]{ClarkZim24}, for $f\in A_{\mathbb{C}}(\mathcal{G})$, the formula $$||f|| = \sup \{||\pi(f)|| \mid \pi \text{ is a representation of } A_{\mathbb{C}}(\mathcal{G})\}$$ is a $C^*$-norm on 
$A_{\mathbb{C}}(\mathcal{G})$.
The {\it groupoid $C^*$-algebra $C^*(\mathcal{G})$} of $\mathcal{G}$ is
 then the completion of $A_{\mathbb{C}}(\mathcal{G})$ with respect to $||\cdot||$ (see~\cite{ClarkZim24} for details).

In order to prove the main result of this section (Theorem~\ref{globalmain}), we divide it into two classes of $L_\K(1,n)$ and  $L_\K(m,n)$, $m>1$.  Recall that Leavitt path algebras, in particular the Leavitt algebras $L_\K(1,n)$, can be modelled by Steinberg algebras~\cite{lisarooz}. 

\begin{thm}\label{main}
Let $\K$ be a field of characteristic zero. Then, the following statements hold:

$(1)$ There is no unital homomorphism from the first Weyl algebra $\mathbb A_\K$ to any unital Steinberg algebra;

$(2)$ Conversely, there is no unital injective homomorphism from a nontrivial unital Steinberg algebra to the first  Weyl algebra $\mathbb A_\K$. 
\end{thm}
\begin{proof}
Assume that there exists a unital ring homomorphism from the first Weyl algebra $\mathbb A_\K$ to the unital Steinberg algebra	$A_\K(\mG)$ of an ample groupoid $\mG$. Then, there exist elements $x, y\in A_\K(\mG)$ such that $xy-yx=1$. Write

\begin{center}
$x = \sum^m_{i=1}a_i1_{U_i}$ and  $y = \sum^n_{j=1}b_j 1_{V_j}$
\end{center}
where $a_i, b_j \in \K\setminus \{0\}$ and $U_i, V_j$ are compact open bisections of $\mG$ for all $1\le i\le m$ and $1\le j\le n$. Let $\K'$ be the subfield of $\K$ generated by all the elements $a_i, b_j$. We then have that $\K'$ is a countable field of characteristic zero, and so the transcendence degree of $\K'$ over $\mathbb{Q}$ is at most $\aleph_0$. Let $S \subset \K'$ and $T \subset \mathbb{C}$ be transcendence bases over $\mathbb{Q}$. Then, there is an injection $S\longrightarrow T$, which extends to an embedding of fields $\mathbb{Q}(S)\longrightarrow \mathbb{Q}(T) \subset \mathbb{C}$. Since $\K'$ is algebraic over $\mathbb{Q}(S)$ and $\mathbb{C}$ is algebraically closed, the embedding extends to an embedding $\K' \longrightarrow \mathbb{C}$. Therefore, we obtain that
 $\K'\subseteq \mathbb C\cap \K$ such that  $xy-yx=1$ in $A_{\K'}(\mG)\subset A_{\mathbb{C}}(\mG)$. Then, since 
 $A_{\mathbb{C}}(\mG)$ is subalgebra of $C^*(\mG)$,
 there is a canonical unital homomorphism $A_{\K'}(\mG) \rightarrow C^*(\mG)$, where $C^*(\mG)$ is the groupoid $C^*$-algebra of $\mG$, and so the Heisenberg equation $xy-yx=1$ can be realized in $C^*(\mG)$. However, by Proposition~\ref{wielandt}, such equation cannot happen in the $C^*$-algebra $C^*(\mG)$, a contradiction. Thus, there is not a unital ring homomorphism from the first Weyl algebra $\mathbb A_\K$ to any unital Steinberg algebra. 

Conversely, suppose there exist a nontrivial unital Steinberg algebra $A_\K(\mG)$ and an injective  homomorphsm of unital rings $\phi: A_\K(\mG) \longrightarrow \mathbb A_\K$. Since the Weyl algebra $\mathbb A_\K$ has no zero divisors,  $A_\K(\mG)$  has no zero divisors, too.  By Proposition \ref{domain}, $\mathcal{G}$ is a torsion-free group with the discrete topology, and $A_\K(\mG)$ is isomorphic to the group algebra $\K\mG$. 

For each $g\in \mG$, we have that $\phi(g)$ is an invertible element in $\mathbb A_\K$. Also, every invertible element of $\mathbb A_\K$ is exactly a nonzero element of $\K$~(see~\cite[Section 2.2]{Cou95}). From these observations, we immediately get that $\K\mG\cong\phi(\K\mG) \subseteq \K$, and since $\phi$ is an injective $\K$-algebra homomorphism, $A_\K(\mG)\cong \K\mG \cong \K$, which contradicts with our hypothesis that $A_\K(\mG)$ is nontrivial. Therefore, there is no unital injective homomorphism from a nontrivial unital Steinberg algebra to $\mathbb A_\K$, thus finishing the proof.
\end{proof}

The algebra $L_\K(1,n)$ can be realized as the Leavitt path   algebra of the rose graph including one vertex and $n$-loops and thus can be modelled as a Steinberg algebra.  However,  the algebras $L_\K(m,n)$, for $1<m<n$, have no zero divisors and cannot be modelled as Steinberg algebras as the following proposition shows.

\begin{prop}
Let $\K$ be a field and $1 < m <n$ positive integers. Then, $L_\K(m, n)$ is not a Steinberg algebra.
\end{prop}
\begin{proof}
Assume that $L_\K(m, n)$ is a Steinberg algebra over $\K$, that means, there exists an ample groupoid $\mathcal{G}$ such that $A_\K(\mathcal{G})\cong L_\K(m, n)$. Since the Leavitt algebra $L_\K(m, n)$  has no zero divisors,  $A_\K(\mG)$  has no zero divisors. By Proposition \ref{domain}, we immediately obtain that $\mathcal{G}$ is a torsion-free group with the discrete topology, and  $L_\K(m,n) \cong A_\K(\mathcal{G})\cong \K\mathcal{G}$, where $\K\mathcal{G}$ is the group algebra of $\mG$. We then have a ring homomorphism $\epsilon: \K\mG \longrightarrow \K$ defined by $\epsilon(\sum_{g\in \mG}k_gg) = \sum_{g\in \mG}k_g$. Since $\K$ is a field and by \cite[Remark 1.5]{Lam99}, $\K\mathcal{G}$ has the IBN property, and so $L_\K(m, n)$ has also the IBN property, a contradiction. This shows that $L_\K(m, n)$ is not isomorphic to any Steinberg algebra, thus finishing the proof.
\end{proof}

We are in a position to present the main result of this section which gives  an affirmative answer to Question~4.7. 

\begin{thm}\label{globalmain}
Let $\K$ be a field of characteristic zero, and  let $m$ and $n$ be positive integers with $m < n$.
Then  Weyl algebra $\mathbb A_\K$ does not embed into $L_\K(m,n)$.
\end{thm}
\begin{proof}
If $m=1$, then  $L_\K(1,n)$ is the Leavitt path algebra associated to a graph with one vertex and $n$ loops (see, e.g., \cite[Proposition 1.3.2]{lpabook}). This algebra can be modelled by a Steinberg algebra (see \cite[Example 3.2]{clarksims}), and hence, by Theorem~\ref{main}, $\mathbb A_\K$ cannot be embedded into $L_\K(1,n)$. 

Next we show that $\mathbb A_\K$ cannot be embedded into $L_\K(m,n)$, where $1 < m < n$, using the $C^*$-algebras $U^{nc}_{(m, n)}$ introduced by McClanahan in \cite{mccal2}. The  $C^*$-algebra $U^{nc}_{(m, n)}$ is generated by elements $u_{ij}$, $1\le i\le m$, $1\le j\le n$, subject to the following relations on $u = [u_{ij}]:$ $u^*u = I_n$ and $uu^* = I_m$, where $I_k$ denotes the identity matrix (see \cite{mccal2}). By the universal property of $L_{\mathbb{C}}(m, n)$, there is a canonical algebra homomorphism $L_{\mathbb{C}}(m, n)\longrightarrow U^{nc}_{(m, n)}$ by the mapping: $1_{L_{\mathbb{C}}(m, n)}\longmapsto 1_{U^{nc}_{(m, n)}}$, $x_{ij}\longmapsto u_{ij}$ and $y_{ji}\longmapsto u^*_{ji}$.

Assume that there exist $x,y\in L_\K(m,n)$, such that $xy-yx=1$. A similar argument as in Theorem~\ref{main} shows that there is a subfield $\K' \subseteq \mathbb C$ such that $xy-yx=1$ in $L_{\K'}(m,n)$.  Since there is always a canonical algebra homomorphism $L_{\mathbb{C}}(m, n)\longrightarrow U^{nc}_{(m, n)}$ and $L_{\K'}(m,n) \subseteq L_{\mathbb{C}}(m,n)$, there is obviously a ring homomorphism  $L_{\K'}(m,n)\rightarrow U^{nc}_{(m, n)}$. Therefore, the Heisenberg equation $xy-yx=1$ can be realized in $U^{nc}_{(m, n)}$. On the other hand, by Proposition~\ref{wielandt}, we always have that  the Heisenberg equation $xy-yx=1$ cannot be realized in $U^{nc}_{(m, n)}$, a contradiction. This implies that $\mathbb A_\K$ cannot be embedded into $L_\K(m,n)$, thus finishing the proof.
\end{proof}
\medskip

\section{graded embedding}
In this section, we consider a graded version of Kirchberg embedding theorem. It is easy to see that in general Leavitt path algebras $L_\K(E)$ of graphs $E$ cannot be graded-embedded into $L_\K(1, 2)$. As an example,  a unital graded homomorphism $L_\K(1,3)\rightarrow L_\K(1,2)$, induces a homomorphism on the level of graded Grothendieck groups~\cite[\S3]{hazi}, 
\begin{align*}
\mathbb Z[1/3] = K_0^{\gr}(L_\K(1,3)) &\longrightarrow K_0^{\gr}(L_\K(1,2)) = \mathbb Z[1/2]\\
1 &\longmapsto 1
\end{align*}
which immediately leads to a contradiction. 

In this final section, we see that if we replace $L_\K(1,2)$ by $L_\K(1,2)\otimes L_\K(1,2)$, we can arrange for a graded embedding of Leavitt path algebras. 

Consider the algebra $L_\K(1,2)\otimes L_\K(1,2)$ with the grading induced by the standard $\mathbb Z$-grading on the first component and the
trivial grading on the second component. We denote the generators of the first component by $x_1,x_2$ and of the second component by $y_1,y_2$. Thus the degree of these elements are $\deg(x_1)=\deg(x_2)=1$ and $\deg(y_1)=\deg(y_2)=0$.

Let $E$ be a row-finite countable graph. Recall (\cite[Theorem 4.1]{brownsor}) that Brownlowe and S{\o}rensen proved that there is a $*$-preserving
embedding
$\varphi \colon L_{\K}(E) \hookrightarrow L_\K(1,2).$

We first establish some useful facts. Recall that in a $*$-algebra $A$, we say $u\in A$ is \emph{unitary} if $uu^*=u^*u=1$ and $p\in A$ is a \emph{projection} if $p=p^*=p^2$. We write $p\sim 1$, if $p=ss^*$ and $1=s^*s$ for some $s\in A$.

\begin{lem}\label{unitary}
The element $ u = x_1 \otimes y_1^* + x_2 \otimes y_2^*$
is a homogeneous of degree $1$ unitary in $L_\K(1,2)\otimes L_\K(1,2)$.
\end{lem}

\begin{proof}
It is clear that the element has degree $1$. We check that it is unitary:
\begin{align*}
( x_1 \otimes y_1^* + x_2 \otimes y_2^* )( x_1 \otimes y_1^* + x_2 \otimes y_2^* )^*
&= ( x_1 \otimes y_1^* + x_2 \otimes y_2^* )( x_1^* \otimes y_1 +  x_2^* \otimes y_2 ) \\
&= x_1  x_1^* \otimes y_1^* y_1
 + x_1  x_2^* \otimes y_1^* y_2\\
&\quad + x_2  x_1^* \otimes y_2^* y_1
 + x_2  x_2^* \otimes y_2^* y_2 \\
&= x_1 x_1^* \otimes 1 + x_2 x_2^* \otimes 1 \\
&= 1.
\end{align*}
The other multiplication is similar.
\end{proof}

\begin{lem}\label{lem4.2}
Let $p_1,\dots,p_n \in L_\K(1,2)$ be pairwise orthogonal projections such that $p_i\sim 1$, $1\leq i\leq n$ and 
\[
\sum_{i=1}^n p_i = 1.
\]
Then, there exists a homogeneous degree $1$ unitary in $L_\K(1,2)\otimes L_\K(1,2)$ commuting with all $1 \otimes p_i$.
\end{lem}

\begin{proof}
Find  $s_i \in L_\K(1,2)$ such that $s_i s_i^* = p_i$ and $s_i^* s_i = 1$.
Let $u=x_1 \otimes y_1^* + x_2 \otimes y_2^*\in L_\K(1,2)$ be the homogeneous unitary of degree 1 from Lemma~\ref{unitary}. Then, the element
\[
t = \sum_{i=1}^n (1 \otimes s_i)\, u\, (1 \otimes s_i^*)
\]
is a homogeneous unitary of degree $1$  commuting with all $1 \otimes p_i$.
\end{proof}

We are now in a position to state the main result of this section.

\begin{thm}\label{grad-embed}
There is a grading-preserving $*$-embedding
\[
\psi \colon L_{\K}(E) \hookrightarrow L_\K(1,2)\otimes L_\K(1,2).
\]
\end{thm}

\begin{proof}
We assume that $E$ has a finite number of vertices, though the result
easily generalizes to countable graphs.

Let $\varphi \colon L_{\K}(E) \hookrightarrow  L_\K(1,2)$ be the 
unital embedding given in~\cite{brownsor}. Let $\{v, e \}_{v \in E^0, \, e \in E^1}$ be the standard
generators of $L_{\K}(E)$. Consider the projections
\[
p_v = \varphi(v) \in  L_\K(1,2) \quad (v \in E^0).
\]

Since the projections $p_v$ satisfy the assumptions of Lemma~\ref{lem4.2} (\cite[Proposition 4.3]{brownsor}), we find a homogeneous degree $1$ unitary $u \in L_\K(1,2)\otimes L_\K(1,2)$ which commutes with all
$1 \otimes p_v$. Define the map $\psi \colon L_{\K}(E) \to L_\K(1,2)\otimes L_\K(1,2)$ on generators by
\[
\psi(v) = 1 \otimes p_v, \qquad
\psi(e) = u (1 \otimes \varphi(e)).
\]

It is straightforward to check that $\psi$ preserves the defining relations
of $L_{\K}(E)$ and is grading-preserving. Injectivity follows from the injectivity
of $\varphi$, since $\psi(L_{\K}(E))$ is contained in the subalgebra generated by
$1 \otimes \varphi(L_{\K}(E))$ and $u$, and $\varphi$ can be recovered by sending
$u \mapsto 1$, thus finishing the proof.
\end{proof}

We finish the note, by pointing out that one can graded-embed Leavitt path algebras $L_\K(E)$ into any strongly  $\mathbb Z$-graded ring $A=\oplus_{i\in \mathbb Z}A_i$, with the monoid $\mathcal V(A_0)=\{0,1\}$. Here $A_0$ is the ring of zero component elements of $A$. For, by Dade’s theorem~\cite[\S1.5]{hazi}, $\mathcal  V^{\gr}(A)=\mathcal  V(A_0)=\{0,1\}$ with the action of $\mathbb Z$ trivially on $\{0,1\}$. So for any conical $\mathbb Z$-monoid $M$, we have an obvious $\mathbb Z$-monoid homomorphism from $M$ to $\mathcal V^{\gr}(A)$. (Here the monoid $\{0,1\}$ with 
$\mathbb Z$ acting trivially is the weak terminal object of conical monoids with $\mathbb Z$-action.) Now the graded version of Bergman's theorem~\cite[Theorem~8.1]{hazlipre} implies that there is a graded homomorphism $B\rightarrow A$, where $B$ is the graded Bergman algebra associated to $M$. 
Specializing to Leavitt path algebras, $L_\K(E)$, 
we get a graded homomorphism $L_\K(E)\rightarrow A$. One can then uses the graded uniqueness theorem to show this homomorphism is injective. As an example one can choose $A=L_\K(1,2)\otimes \K[x,x^{-1}]$, where the first component has given the trivial grading and conclude that for any finite graph $E$, we have $L_\K(E) \hookrightarrow A$.

{\bf Acknowledgement.} 
Hazrat acknowledges Australian Research Council Discovery Grant DP230103184. He would like to thank Ralf Meyer for hosting him in G\"ottingen in June 2025, where the initial discussions leading to this work began.
Bilich is partially supported by the Bloom PhD scholarship and the DFG Middle-Eastern collaboration project no. 529300231. Nam is supported by the Vietnam Academy of Science and Technology.
The authors express their gratitude to Gene Abrams, Guillermo Corti\~nas and Ralf Meyer for their valuable suggestions which led to the final shape of the paper.


\begin{thebibliography}{}


\bibitem{GenePino} G. Abrams, G. Aranda-Pino, {\it The Leavitt path algebra of a graph}, J. Algebra {\bf 293} (2005), 319-334.

\bibitem{Genenotices} G. Abrams, {\it What is ... a Leavitt path algebra?}, Notices Amer. Math. Soc. {\bf 63} (2016), no. 8, 910--911.


\bibitem{lpabook} G. Abrams, P. Ara, M. Siles Molina, Leavitt Path Algebras, {\it Lecture Notes in Mathematics}, Volume \textbf{2191}, Springer Verlag, 2017.


\bibitem{ANP2017} G. Abrams, T. G Nam, N. T. Phuc, {\it Leavitt path algebras having unbounded generating number,}  J. Pure Appl. Algebra, {\bf 221} (2017), 1322-1343.

\bibitem{GeneRooz} G. Abrams, R. Hazrat, Connections between Abelian sandpile models and the K-theory of weighted Leavitt path algebras,  {\it Eur. J. Math.} {\bf 9} (2023), no. 2, Paper No. 21, 28 pp.


\bibitem{Pino}
G. Aranda Pino, J. Clark, A. an Huef, I. Raeburn: \emph{Kumjian-Pask algebras of higher-rank graphs}, Trans. Amer. Math. Soc. \textbf{365} (2013), 3613--3641.

\bibitem{AraGoodearl}  P. Ara, K. R. Goodearl, {\it Leavitt path algebras of separated graphs},  J. Reine Angew. Math. {\bf 669} (2012), 165-224.


\bibitem{arawillie}  P. Ara, G. Corti\~ nas, {\it Tensor products of Leavitt path algebras,} Proc. Amer. Math. Soc. {\bf 141} (2013), no. 8, 2629--2639. 

\bibitem{ara2006} P. Ara, M.A. Moreno, E. Pardo, {\it Nonstable K-theory for graph algebras}, Algebras and Representation Theory
$\mathbf{10}$ (2007), 157--178.




\bibitem{Berg74}  G.M. Bergman, {\it Coproducts and some universal ring constructions,}  Trans. Amer. Math. Soc. {\bf 200} (1974), 33--88.


\bibitem{brownsor} N. Brownlowe, A.P.W. S\o rensen, {\it Leavitt $R$-algebras over countable graphs embed
into $L_{2,R}$}, J. Algebra {\bf 454} (2016) 334--356.

\bibitem{sorensen}  N. Brownlowe, A.P.W. S\o rensen, {\it $L_{2,\mathbb Z} \otimes L_{2,\mathbb Z}$  does not embed in $L_{2,\mathbb Z}$}, J. Algebra  {\bf 456} (2016) 1--22.  

\bibitem{clarksims} L.O. Clark, A. Sims, {\it Equivalent groupoids have Morita equivalent Steinberg algebras},  J. Pure Appl. Algebra {\bf 219} (2015) 2062-2075.

\bibitem{lisarooz} L.O. Clark, R. Hazrat, {\it \'Etale groupoid and Steinberg algebras, a concise introduction}, Leavitt Path Algebras and Classical K-Theory,  A.A. Ambily, R. Hazrat, and B. Sury eds. Indian Statistical Institute Series, Springer (2020). 

\bibitem{ClarkZim24} L.O. Clark and J. Zimmerman, {\it A Steinberg algebra approach to \'etale groupoid $C^*$-algebras},  J. Operator Theory {\bf 91} (2024),  349-371.

\bibitem{Cortinas2020} G. Corti\~ nas, D. Montero, {\it Homotopy classification of Leavitt path algebras,} Adv. Math. {\bf 362} (2020), 106961. 

\bibitem{Cou95}
S.C. Coutinho, {\it A Primer of Algebraic D-Modules}, London Mathematical Society Student Texts, Volume {\bf 33}, Cambridge University Press, 1995.

\bibitem{HajLiu} P. M. Hajac and Y. Liu, {\it Unital embedding of Cuntz algebras from path homomorphisms of graphs}, arXiv:2412.17167.

\bibitem{halmos} P. Halmos,  A Hilbert Space Problem Book,  {\it Graduate Texts in Mathematics}, Volume {\bf 19},  Springer-Verlag, Berlin, 1982,

\bibitem{hazi} R. Hazrat,  Graded rings and graded Grothendieck groups, \emph{London Mathematical Society Lecture
Note Series} {\bf 435}, Cambridge University Press, Cambridge, 2016.

\bibitem{hazlipre} R. Hazrat, H. Li, R. Preusser, \emph{Bergman algebras: The graded universal algebra constructions}, 	arXiv:2403.01703 [math.RA]


\bibitem{howie} M.~Howie, Fundamentals of Semigroup Theory,
Oxford University Press, 1995, 


\bibitem{Kawa} K. Kawamuar, {\it Universal algebra of sectors},  Internat. J. Algebra Comput. \textbf{19} (2009), 347--371. 

\bibitem{Lam99}
T.Y. Lam, Lectures on Modules and Rings, Springer-Verlag, New York-Berlin, 1999


\bibitem{leavitt1} W.G. Leavitt, {\it Modules without invariant basis number}, Proc. Amer. Math. Soc. {\bf 8} (1957), 322--328.

\bibitem{leavitt2}W.G.  Leavitt, {\it The module type of a ring,} Trans. Amer. Math. Soc. {\bf 103} (1962) 113--130. 




\bibitem{mccal2}  K. McClanahan, \emph{$K$-theory and Ext-theory for unitary $C^*$-algebras}, Rocky Mountain J. Math. {\bf 23} (1993) 1063--1080.

\bibitem{p:gisatoa} A.L.T. Paterson,  \textit{Groupoids, Inverse Semigroups, and their Operator Algebras}, vol. 170 of Progress in Mathematics. Springer, 1999.

\bibitem{renault}
J. Renault, {\it A Groupoid Approach to $C^*$-Algebras}, Lecture Notes in Mathematics, vol. 793, Springer, Berlin, 1980.
 
\bibitem{r:egatq}  P. Resende, {\it \'{E}tale groupoids and their quantales},  Adv. Math. \textbf{208} (2007), 147--209.

\bibitem{rss2015}  E. Ruiz, A. Sims, P. W. S\o rensen, {\it UCT-Kirchberg algebras have nuclear dimension one},  Adv. Math. \textbf{279} (2015), 1--28.




\bibitem{stein:agatdisa}
B. Steinberg, {\it A groupoid approach to discrete inverse semigroup algebras}, Adv. Math. \textbf{223} (2010), 689--727.

\bibitem{steinwyk}
B. Steinberg and D. W. van Wyk, {\it On von Neumann regularity of ample groupoid algebras}, 	arXiv:2401.03308.


\bibitem{Wielandt} H. Wielandt, {\it \"Uber die Unbeschr\"anktheit der Operatoren der Quantenmechanik}, Math. Ann. {\bf 121} (1949) 21 pages. 


\end{thebibliography}
\end{document}